\documentclass[12pt,a4paper]{article}
\usepackage[T1]{fontenc}
\usepackage{lmodern}
\usepackage[latin1]{inputenc}
\usepackage{latexsym}                 
\usepackage{color}                 
\usepackage{amssymb}
\usepackage{amsmath}
\usepackage{amsthm}
\usepackage{graphicx}
\usepackage{hyperref}
\usepackage{verbatim}
\usepackage{mathptmx}      % use Times fonts if available on your TeX system
\def\eref#1{(\ref{#1}%
)}
\def\RSref#1{\ref{#1}%
}
\def\RSlabel#1{\label{#1}%
}
\def\RScite#1{\cite{#1}%
}

\newcommand{\bql}[1]{%
\begin{equation}\label{#1}%
}

\def\filename#1{}

\newcommand{\eq}{\end{equation}}
\def\fa{\hbox{ for all }}

\def\dfrac#1#2{\displaystyle{\frac{#1}{#2}   }}

\def\b1{\mathbf 1}

\newcommand{\R}{\ensuremath{\mathbb{R}}}
\def\boc{\mathbf{c}}	
\def\boe{\mathbf{e}}	
\def\boA{\mathbf{A}}	
\def\boB{\mathbf{B}}	
\def\boC{\mathbf{C}}	
	
\def\bou{\mathbf{u}}	
\def\bof{\mathbf{f}}	
\def\bog{\mathbf{g}}	
	
\def\boI{\mathbf{I}}

\def\bor{\mathbf{r}}

\def\bov{\mathbf{v}}

\def\biglf{\par\bigskip\noindent}
\newcommand{\red}[1]{{\color{red} #1 }}     % RS
\newtheorem{definition}{Definition}

\newtheorem{theorem}{Theorem}

\begin{document}
\begin{center}
{\bf Error Analysis of Nodal Meshless Methods}
\biglf 
Robert Schaback \\
Univ. Göttingen\\
schaback@math.uni-goettingen.de\\
http://num.math.uni-goettingen.de/schaback/research/group.html
\biglf
Dedicated to the memory of Ted Belytschko
\biglf
Version of Nov. 11, 2015
\end{center}

{\bf Abstract:} There are many application papers that
solve elliptic boundary value problems by meshless methods, and they use 
various forms of generalized stiffness matrices that 
approximate derivatives of functions
from values at scattered nodes
$x_1,\ldots,x_M\in \Omega\subset\R^d$. 
If $u^*$ is the true solution in some Sobolev space $S$ allowing
enough smoothness for the problem in question, 
and if the calculated approximate values at the nodes are
denoted by $\tilde u_1,\ldots,\tilde u_M$, the canonical form of error bounds
is 
$$
\max_{1\leq j\leq M}|u^*(x_j)-\tilde u_j|\leq \epsilon \|u^*\|_S
$$
where $\epsilon$ depends crucially on the problem and the discretization,
but not on the solution. 
This 
contribution shows how to calculate 
such $\epsilon$ {\em numerically and explicitly}, for any 
sort of discretization of strong problems via nodal values, may
the discretization use Moving Least Squares, unsymmetric or symmetric
RBF collocation, or localized RBF or polynomial stencils.
This allows users to compare different discretizations
with respect to error bounds of the above form, without knowing exact
solutions, and admitting all possible ways to
set up generalized stiffness matrices. The error analysis is proven to be sharp 
under mild additional assumptions. As a byproduct, it allows to
construct worst cases that push discretizations to their limits.
All of this is illustrated by numerical examples.

%****************************************************************
\section{Introduction}\RSlabel{SecIntro}
Following the seminal survey \cite{belytschko-et-al:1996-1}
by Ted Belytschko et.al. in 1996,  meshless methods for PDE solving often work 
``{\it entirely in terms of values at nodes}''.
This means that large linear systems 
are set up that have values
$u(x_1),\ldots,u(x_M)$ of an unknown function $u$ as unknowns,
while the equations model the underlying PDE problem in  discretized way.
Altogether, the discrete problems have the form
\bql{eqnaivesys}
\displaystyle{\sum_{j=1}^M a_{kj}u(x_j)\approx f_k,\;1\leq k\leq N   } 
\eq
with $N\geq M$, whatever the underlying PDE problem is, and the 
$N\times M$ matrix $\boA$ with entries $a_{kj}$ 
can be called a {\em generalized stiffness matrix}.
\biglf
Users solve the system somehow and 
then get values $\tilde u_1,\ldots,\tilde u_M$  that satisfy 
$$
\displaystyle{\sum_{j=1}^M a_{kj}\tilde u_j\approx f_k,\;1\leq k\leq N   }, 
$$
but they should know how far these values are from the values
$u^*(x_j)$ of the true solution of the PDE problem that is supposed to exist.
\biglf
The main goal of this paper
is to provide tools that allow users to assess the
quality of their discretization, no matter
how the problem was discretized or how the system was actually solved. 
The computer should tell the user whether the 
discretization
is useful or not. It will turn out that this is possible, and at 
tolerable 
computational cost that is proportional to the 
complexity for setting up the system, not for solving it.
%and certain instabilities that will hopefully be overcome by
%future research.
\biglf
The only additional ingredient is a specification of 
the smoothness of the true
solution $u^*$, and this is done in terms of a strong norm 
$\|.\|_S$, e.g. a higher-order Sobolev norm or seminorm. The whole problem
will then be implicitly scaled by $\|u^*\|_S$, and we assert an
absolute bound of the form 
$$
\max_{1\leq j\leq M}|u^*(x_j)-\tilde u_j|\leq \epsilon \|u^*\|_S
$$
or a relative bound
$$
\dfrac{\max_{1\leq j\leq M}|u^*(x_j)-\tilde u_j|}{\|u^*\|_S}\leq \epsilon
$$
with an entity $\epsilon$ that can be calculated. It will be a product of 
two values caring for 
{\em stability}  and {\em consistency}, respectively, and these are calculated
and analyzed separately. 
\biglf
Section \RSref{SecPro} will set up the large range of PDE 
or, more generally, operator equation problems we are 
able to handle, and Section \RSref{SecEA} provides the
backbone of our error analysis. It must naturally contain some
versions of {\em consistency} and {\em stability}, and we deal 
with these in
Sections   \RSref{SecCA} and \RSref{SecSA}, with an interlude on 
polyharmonic kernels in Section 
\RSref{SecAbPHK}. For given Sobolev smoothness order
$m$, these provide stable, sparse, and error-optimal nodal approximations
of differential operators.   
Numerical examples
follow in Section \RSref{SecEx}, demonstrating how to work with the tools of
this paper. It turns out that the evaluation of stability is 
easier than expected, while the evaluation of consistency 
often suffers from severe numerical cancellation that is to be overcome
by future research, or that is avoided by using special 
scale-invariant approximations,
e.g. via polyharmonic kernels along the lines of Section 
\RSref{SecAbPHK}.
%****************************************************************
\section{Problems and Their Discretizations}\RSlabel{SecPro}
We have to connect the system
\eref{eqnaivesys} back to the original PDE problem, and we 
do this in an unconventional but useful way that we use  
successfully 
since \RScite{schaback-wendland:1999-1} in 1999.
%%%%%%%%%%%%%%%%%%%%%%%%%
\subsection{Analytic Problems}\RSlabel{SecAnaPro}
For example, 
consider a model boundary value problem of the form
\bql{eqDP}
\begin{array}{rcll}
L u &=& f & \hbox{ in } \Omega\subset\R^d\\ 
B u &=& g & \hbox{ in } \Gamma:=\partial \Omega 
\end{array}
\eq
where $f, \;g$ are given functions on $\Omega$ and $\Gamma$, respectively, 
and $L,\;B$ are linear operators, defined and continuous
on some normed linear space $U$ in which the true solution $u^*$ should lie.
Looking closer, this is
an infinite number of linear constraints
$$ 
\begin{array}{rcll}
L u(y) &=& f(y) & \fa y\in \Omega\subset\R^d\\ 
B u (z) &=& g(z)  & \fa z\in \Gamma:=\partial \Omega 
\end{array} 
$$
and these can be generalized as infinitely many 
linear functionals acting on the function $u$, namely
\bql{eqConPro}
\lambda(u)=f_\lambda \fa \lambda \in \Lambda\subset U^*
\eq
where the set $\Lambda$ is contained in the topological dual $U^*$ of $U$,
in our example
\bql{eqlamLmuB}
\Lambda=\{\delta_y\circ L,\;y\in \Omega  \}
\cup \{\delta_z\circ B,\;z\in \Gamma  \}.
\eq
\begin{definition}\RSlabel{DefProb}
An admissible problem in the sense of this paper
consists in finding an 
$u$ from some normed linear space 
$U$ such that \eref{eqConPro} holds 
for a fixed set $\Lambda\subset U^*$.
Furthermore, solvability via $f_\lambda=\lambda(u^*)
\fa \lambda \in \Lambda\subset U^*$ for some $u^*\in U$
is always assumed. 
\end{definition} 
Clearly, this allows various classes
of differential equations and boundary
conditions. in weak or strong form. For examples, see \RScite{schaback:2015-2}.
Here, we just mention that the standard functionals for weak problems with
$L=-\Delta$ are of the form 
\bql{eqlamweak}
\lambda_v(u):=\int_\Omega (\nabla u)^T \nabla v
\eq
where $v$ is an arbitrary test function from $W_0^1(\Omega)$.
%****************************************************************
\subsection{Discretization}\RSlabel{SecDis}
The connection of the problem \eref{eqConPro} 
to the discrete linear system \eref{eqnaivesys}
usually starts with specifying a 
finite subset $\Lambda_N=\{\lambda_1,\ldots,\lambda_N\}\subset\Lambda$
of {\em test} functionals. But then it splits into  
two essentially different branches. 
\biglf
The {\em shape function} approach defines functions
$u_j\;:\;\Omega\to \R$ with the Lagrange property
$u_i(x_j)=\delta_{ij},\;1\leq i,j\leq M$ and defines the elements
$a_{kj}$ of the stiffness matrix as
$a_{kj}:=\lambda_k(u_j)$. This means that 
the application of the functionals $\lambda_k$ on {\em trial functions}
$$
u(x)=\sum_{j=1}^Mu(x_j)u_j(x)
$$ 
is exact, and the linear system \eref{eqnaivesys} describes the exact 
action of the selected test functionals on the trial space.
Typical instances of the shape function approach 
are standard applications
of Moving Least Squares (MLS) trial functions
\RScite{wendland:2000-1,armentano:2001-1,armentano-duran:2001-1}.
Such applications were surveyed in 
\RScite{belytschko-et-al:1996-1} and incorporate many versions of
the Meshless Local Petrov Galerkin (MLPG) technique
\RScite{atluri-zhu:1998-1}. Another popular shape function method is
unsymmetric or symmetric kernel-based collocation, see % starting with
\RScite{kansa:1986-1,fasshauer:1997-1,%
franke-schaback:1998-2a,franke-schaback:1998-1}.
\biglf
But one can omit shape functions completely, at the cost of
sacrificing exactness. Then the selected functionals $\lambda_k$
are each approximated by linear combinations of the functionals
$\delta_{x_1},\ldots,\delta_{x_M}$ by requiring
\bql{eqlamku},
\lambda_k(u)\approx \displaystyle{\sum_{j=1}^M a_{kj}\delta_{x_j}u=
\sum_{j=1}^M a_{kj}u(x_j),\;1\leq k\leq
  N,\;\fa u\in U.  } 
\eq
This approach can be called {\em direct} discretization, because it
bypasses shape functions. It is the standard technique for 
{\em generalized
finite differences} (FD) \RScite{YC:MG80}, and it comes up again in meshless
methods at many places, starting with \RScite{nayroles-et-al:1981,tolstykh:2003-1}
and called {\em RBF-FD} 
or {\em local RBF collocation} by various authors, e.g.
\RScite{wright-fornberg:2006-1,flyer-et-al:2015-1}.
The generalized finite difference approximations 
may be calculated via radial kernels using 
local selections of nodes only
\RScite{sarler:2007-1,yao-et-al:2010-1,yao-et-al:2011-1},
and there are papers on how to calculate such approximations, e.g.
\RScite{davydov-schaback:2015-1,larsson-et-al:2013-1}.
Bypassing Moving Least Squares trial functions, direct methods
in the context of Meshless Local Petrov Galerkin
techniques are in \RScite{mirzaei-et-al:2012-1,mirzaei-schaback:2013-1},
connected to {\em diffuse derivatives} \RScite{nayroles-et-al:1981}. 
For a mixture of kernel-based and MLS techniques, see
\RScite{kim-kim:2003-1}. 
\biglf
This contribution will work in both cases, with a certain preference for the 
direct approach. The paper \RScite{schaback:2015-2} 
focuses on shape function methods instead. It proves that uniform stability
can be achieved for all well-posed problems by choosing a suitable
discretization, and then convergence can be inferred from
standard convergence rates of approximations of derivatives of the true solution
from derivatives of trial functions.  The methods of \RScite{schaback:2015-2}
fail for direct methods, and this was the main reason to write this paper.
%****************************************************************
\subsection{Nodal Trial Approximations}\RSlabel{SecNotNodApp}
In addition to Definition
\RSref{DefProb} 
we now assume that $U$ is a space of functions on some set $\Omega$, and that 
point evaluation is continuous, i.e. $\delta_x\in U^* \fa x\in\Omega$.
We fix a finite set $X_M$ of $M$ {\em nodes}  $x_1,\ldots,x_M$ and 
denote the span of the functionals $\delta_{x_j}$ by $D_M$. 
\biglf
For each $\lambda\in \Lambda$ we consider a linear approximation 
$\tilde \lambda$ 
to $\lambda$ from $D_M$, i.e.
\bql{eqlamalu}
\lambda(u) \approx \sum_{j=1}^M a_j(\lambda)u(x_j)=:\tilde \lambda (u) 
\eq
Note that there is no trial space of functions, and no {\em shape 
functions} at all, just {\em nodal values} and approximations of functionals
from nodal values. It should be clear how the functionals in \eref{eqlamLmuB} 
can be approximated as in \eref{eqlamalu} via values at nodes. 
\biglf
In the sense of the preceding section, this looks like a {\em direct}
discretization, but it also covers the {\em shape function} approach, because
it is allowed to take $a_j(\lambda)=\lambda(u_j)$ for shape functions $u_j$ with
the Lagrange property. 
%****************************************************************
\subsection{Testing}\RSlabel{SecTest}
Given a nodal trial approximation, consider a finite subset $\Lambda_N$
of functionals $\lambda_1,\ldots,\lambda_N$ and pose the
possibly overdetermined linear system
\bql{eqLinSys}
\lambda_k(u^*)=f_{\lambda_k}=
\displaystyle{\sum_{j=1}^M a_j(\lambda_k)u_j   } 
\eq
for unknown {\em nodal values} $u_1,\ldots,u_M$ that may be interpreted as
approximations to $u^*(x_1),\ldots,u^*(x_M)$.  We 
call $\Lambda_N$ a {\em test selection} of functionals, and remark that 
we have obtained a system 
of the form \eref{eqnaivesys}. 
\biglf
For what follows, we write the linear system \eref{eqLinSys} 
in matrix form als
\bql{eqfAu}
\bof=\boA \bou
\eq 
with 
$$ 
\begin{array}{rcll}
\boA&=&(a_j(\lambda_k))_{1\leq k\leq N,1\leq j\leq M}&\in \R^{N\times M}\\
\bof&=&(f_{\lambda_1},\ldots,f_{\lambda_N})^T&\in \R^N\\
\bou&=&(u(x_1),\ldots,u(x_M))^T&\in \R^M.
\end{array}
$$
Likewise, we denote the vector of exact nodal values $u^*(x_j)$ by 
$\bou^*$, and $\tilde \bou$ will be the vector of nodal values $\tilde u_j$ 
that
is obtained by some numerical method that solves the system \eref{eqLinSys}
approximately. 
\biglf
It is well-known \RScite{hon-schaback:2001-1} that square systems 
of certain meshless methods may be singular, but it is also known
\RScite{schaback:2015-2} that one can bypass that problem by 
{\em overtesting}, i.e. choosing $N$ larger than $M$.
This leads to overdetermined systems, but they can be
handled by standard methods like the MATLAB backslash 
in a satisfactory way. Here, we expect that users set up their 
$N\times M$ stiffness matrix $\boA$ by sufficiently thorough testing,
i.e. by selecting many test functionals $\lambda_1,\ldots,\lambda_N$ 
so that the matrix has rank $M\leq N$. Section \RSref{SecSA} will show that
users can expect good stability if they handle
a well-posed problem with sufficient overtesting. 
Note further that for cases like the standard Dirichlet
problem \eref{eqDP}, the set $\Lambda_N$ has to contain a reasonable mixture
of functionals connected to the differential operator and functionals connected
to boundary values. Since we focus on general 
worst-case error estimates here,
insufficient overtesting and an unbalanced mixture of boundary and differential
equation approximations will result in error bounds that either 
cannot be calculated due to rank loss
or come out large. The computer should reveal 
whether a discretization is good or
not.
%****************************************************************
\section{Error Analysis}\RSlabel{SecEA}
The goal of this paper is to derive 
useful bounds for 
$\|\bou^*-\tilde \bou\|_\infty$, but we do not care 
for an error analysis away from the nodes. Instead, 
we assume a {\em postprocessing step}
that interpolates the elements of $\tilde \bou$ to generate an approximation 
$\tilde u$ to the solution $u^*$ in the whole domain. Our analysis will
accept any numerical solution $\tilde \bou$ in terms of nodal values
and provide an error bound with small additional computational effort.
%****************************************************************
\subsection{Residuals}\RSlabel{Secres}
We start with evaluating the {\em residual}  $\bor:=\bof-\boA\tilde\bou\in \R^N$
no matter how the numerical solution $\tilde \bou$ was obtained. 
This can be explicitly done except for roundoff errors, and needs no 
derivation of upper bounds. Since in general  
the final error at the nodes will be larger than the 
observed residuals, users should refine their discretization when they 
encounter residuals that are very much larger than the expected error in the
solution.
%****************************************************************
\subsection{Stability}\RSlabel{SecStabEx}
In Section \RSref{SecTest} we postulated that
users calculate an 
$N\times M$ stiffness matrix $\boA$ that has no rank loss.
Then the {\em stability constant}
\bql{eqCSC}
C_S(\boA):=\displaystyle{\sup_{\bou\neq 0}
\dfrac{\|\bou\|_p}{\|\boA\bou\|_q}   }
\eq
is finite for any choice of discrete norms $\|.\|_p$ and $\|.\|_q$
on $\R^M$ and $\R^N$, respectively, with $1\leq p,q\leq \infty$
being fixed here, and dropped from the notation.
In principle, this constant can be explicitly calculated for standard norms,
but we refer to Section \RSref{SecSA} on 
how it is treated in theory and practice. We shall mainly focus on 
well-posed cases where
$C_S(\boA)$
can be expected to be reasonably bounded,
while norms of $\boA$ get very large. This implies that the ratios 
$\|\bou\|_p/\|\boA \bou\|_q$ can vary in a wide range limited by
\bql{eqrange}
\|\boA\|^{-1}_{q,p}\leq \dfrac{\|\bou\|_p}{\|\boA \bou\|_q}\leq C_S(\boA).
\eq
If we assume that we can deal with the stability constant $C_S(\boA)$,
the second step of error analysis is
\bql{eqfirststep}
\begin{array}{rcl}
\|\bou^*-\tilde\bou\|_p 
&\leq& 
C_S(\boA)\|\boA(\bou^*-\tilde\bou)\|_q\\
&\leq& 
C_S(\boA)(\|\boA\bou^*-\bof\|_q+
\|\bof-\boA\tilde\bou\|_q)\\
&\leq& 
C_S(\boA)(\|\boA\bou^*-\bof\|_q+
\|\bor\|_q)\\
\end{array}
\eq
and we are left to handle the  
{\em consistency term} $\|\boA\bou^*-\bof\|_q$
that still contains the unknown true solution $\bou^*$.
Note that $\bof$ is not necessarily in the range of $\boA$,
and we cannot expect to get zero residuals $\bor$.
%\biglf
%Note that in unlucky cases the right-hand side of
%\eref{eqfirststep} can be like
%$\|\boA\|_{q,p}\|\bou^*-\tilde\bou\|_p$ with a rather large norm $\|\boA\|_{q,p}$.
%****************************************************************
\subsection{Consistency}\RSlabel{SecConS}
For all approximations \eref{eqlamalu} 
we assume that there is a {\em consistency} error bound
\bql{eqconserr}
|\lambda(u)-\tilde \lambda (u)|\leq c(\lambda)\|u\|_S
\eq
for all $u$ in some {\em regularity} subspace $U_S$ of $U$
that carries a strong norm or seminorm $\|.\|_S$.  
In case of a seminorm, we have to assume that the approximation $\tilde \lambda$
is an exact approximation to $\lambda$ on the nullspace of the semminorm,
but we shall use seminorms only in Section \RSref{SecAbPHK} below. 
If the solution $u^*$ 
has plenty of smoothness, one may expect
that $c(\lambda)\|u^*\|_S$ is small, provided that the 
discretization quality keeps up with the smoothness. 
In section \RSref{SecCA},
we shall consider cases
where the $c(\lambda)$ can be calculated explicitly.
\biglf
The bound \eref{eqconserr} now specializes to 
$$
\|\boA\bou^*-\bof\|_q\leq \|\boc\|_q \|u^*\|_S
$$
with the vector
$$
\boc=(c(\lambda_1),\ldots,c(\lambda_N))^T\in \R^N,
$$
and the error in \eref{eqfirststep} is bounded absolutely by
$$
\|\bou^*-\tilde\bou\|_p 
\leq C_S(\boA)\left(\|\boc\|_q \|u^*\|_S+
\|\bor\|_q\right)
$$
and relatively by
\bql{eqsecondstep}
\dfrac{\|\bou^*-\tilde\bou\|_p}{\|u^*\|_S} 
\leq C_S(\boA)\left(\|\boc\|_q +
\dfrac{\|\bor\|_q}{\|u^*\|_S}\right).
\eq
This still contains the unknown solution $u^*$. But in kernel-based 
spaces, there are ways to get estimates of $\|u^*\|_S$ via
interpolation. A strict but costly way is to interpolate the data vector $\bof$
by symmetric kernel collocation to get a function $u^*_\bof$ with
$\|u^*_\bof\|_S\leq \|u^*\|_S$, and this norm can be plugged into
\eref{eqsecondstep}. In single applications, users would prefer to
take the values of $u^*_\bof$ in the nodes as results, since they are known
to be error-optimal \RScite{schaback:2015-3}. But if discretizations with
certain given matrices $\boA$ are to be evaluated or compared, this suggestion makes
sense to get the right-hand side of \eref{eqsecondstep}
independent of $u^*$.
%****************************************************************
\subsection{Residual Minimization}\RSlabel{SecResMin}
To handle the awkward final term in \eref{eqsecondstep}
without additional calculations,
we impose a rather weak additional condition on the numerical procedure
that produces $\tilde \bou$  as an approximate solution to \eref{eqfAu}.
In particular, we require
\bql{eqAufKAuf}
\|\boA\tilde \bou-\bof\|_q \leq K(\boA)\|\boA\bou^*-\bof\|_q,
\eq
which can be obtained with $K(\boA)=1$ if $\tilde \bou$
is calculated via minimization of the residual  $\|\boA\bou-\bof\|_q$
over all $\bou\in\R^M$, or with $K(\boA)=0$ if $\bof$ is in the range
of $\boA$. Anyway, we assume that users have a way to
solve the system \eref{eqfAu} approximately such that
\eref{eqAufKAuf} holds with a known and moderate constant $K(\boA)$.
\biglf
Then \eref{eqAufKAuf} implies
$$
\begin{array}{rcl}
\|\bor\|_q &=& \|\boA\tilde \bou-\bof\|_q \\
&\leq & 
K(\boA)\|\boA\bou^*-\bof\|_q\\
&\leq & 
K(\boA)\|\boc\|_q\|u^*\|_S\\
\end{array}
$$
and bounds $\|\bor\|_q$ in terms of $\|u^*\|_S$.
%****************************************************************
\subsection{Final Relative Error Bound}\RSlabel{SecFREB}
\begin{theorem}\RSlabel{the3step}
Under the above assumptions,
\bql{eqthirdstep}
\dfrac{\|\bou^*-\tilde\bou\|_p}{\|u^*\|_S} 
\leq (1+K(\boA))C_S(\boA)\|\boc\|_q.
\eq
\end{theorem} 
\begin{proof}
We can insert \eref{eqAufKAuf} directly into 
\eref{eqfirststep} to get
$$
\begin{array}{rcl}
\|\bou^*-\tilde\bou\|_p 
&\leq & 
C_S(\boA)(1+K(\boA))\|\boA\bou^*-\bof\|_q\\
&\leq & 
(1+K(\boA))C_S(\boA)\|\boc\|_q\|u^*\|_S\\
\end{array}
$$
and finally \eref{eqthirdstep},
where now all elements of the right-hand side are accessible. 
\end{proof}
This is as far as one can go, not having any additional
information on how $u^*$ scales. The final form of 
\eref{eqthirdstep} shows the classical elements
of convergence analysis, since the right-hand side consists of
a {\em stability} term $C_S(\boA)$ and 
a {\em consistency} term $\|\boc\|_q$.   The factor 
$1+K(\boA)$ can be seen as a {\em computational accuracy} term. 
\biglf
Examples in Section \RSref{SecEx} will show how these 
relative error bounds work in practice. Before that,
the next sections will demonstrate theoretically
why users can expect that the ingredients of 
the bound in \eref{eqthirdstep} can be expected to be small. 
For this analysis, we shall 
assume that users know which regularity the true solution has,
because we shall have to express
everything in terms of $\|u^*\|_S$. 
\biglf
At this point, some remarks on error bounds 
should be made, because papers focusing on applications of meshless methods 
often contain one of the two standard crimes of error assessment.
\biglf
The first is to take a problem with a known solution $u^*$ 
that supplies the data, calculate nodal values $\tilde \bou$ 
by some hopefully new method and then
compare with $\bou^*$ to conclude that the method is good because
$\|\bou^*-\tilde \bou\|$ is small. But the method may be 
intolerably unstable. If the input is changed
very slightly, it may produce a seriously different numerical solution
$\hat\bou$ that reproduces the data as well as  $\tilde \bou$. 
The ``quality'' of the result $\tilde\bou$ may be just lucky,
it does not prove anything about the method used.
\biglf
The second crime, usually committed when there is no explicit
solution known, is to evaluate residuals $\bor=\boA\tilde\bou-\bof$ 
and to conclude that  $\|\bou^*-\tilde \bou\|$ is small
because residuals are small. This also ignores stability.
There even are papers that claim convergence of methods 
by showing that residuals converge to zero when the discretization is refined.
This reduces convergence rates of a PDE solver to 
rates of consistency, again ignoring stability problems that may counteract
against good consistency. Section \RSref{SecEx} will
demonstrate this effect by examples.
\biglf
This paper will avoid these crimes, but on the downside 
our error analysis is a worst-case theory that will necessarily overestimate 
errors of single cases.
%%%%%%%%%%%%%%%%%%%%%%%
%****************************************************************
\subsection{Sharpness}\RSlabel{SecSh}
In particular, if users take a specific problem \eref{eqDP} 
with data functions $f$ and $g$ and a known solution $u^*$, and if they
evaluate the observed error and the bound
\eref{eqthirdstep}, they will often see quite an overestimation of the error. 
This is due to the fact that they have a special case that is far away from
being worst possible for the given PDE discretization, and this is 
comparable to a lottery win, as we shall prove now.
\begin{theorem}\RSlabel{thesharp} 
For all $K(\boA)>1$ there is some $u^*\in U_S$ 
and an admissible solution vector $\tilde\bou$ satisfying
\eref{eqAufKAuf} such that
\bql{eqsharp}
(K(\boA)-1)C_S(\boA)\|u^*\|_S\|\boc\|_\infty\leq
\|\bou^*-\tilde\bou\|_\infty\leq (K(\boA)+1)C_S(\boA)\|u^*\|_S\|\boc\|_\infty
\eq
showing that the above worst-case 
error analysis cannot be improved much.
\end{theorem} 
\noindent\begin{proof}
We first take the worst possible value vector $\bou_S$ for stability, satisfying
$$
\|\bou_S\|_\infty=C_S(\boA)\|\boA\bou_S\|_\infty
$$
and normalize it to $\|\bou_S\|_\infty=1$. Then we consider the worst case of
consistency, and we go into a kernel-based context.
\biglf
Let the consistency vector $\boc$ 
attain its norm at some index $j,\;1\leq j\leq N$, i.e.
$\|\boc\|_\infty=c(\lambda_j)$. Then there is a function $u_j\in U_S$  with
$$
|\lambda_j(u_j)-\tilde\lambda_j(u_j)|=c(\lambda_j)\|u_j\|_S=c(\lambda_j)^2,
$$
namely by taking the Riesz representer $u_j:=(\lambda_j-\tilde
\lambda_j)^xK(x,\cdot)$
of the error functional.
The values of $u_j$ at the
nodes form a vector $\bou_j$, and we take the data $f$ as exact values
of $u_j$, i.e. $f_k:=\lambda_k(u_j),\;1\leq k\leq N$ to let $u_j$ play the role
of the true solution $u^*$, in particular $\bou^*=\bou_j$ and
$\|u^*\|_S=\|u_j\|_S=c(\lambda_j)=\|\boc\|_\infty$. 
\biglf
We then define
$\tilde \bou:=\bou^*+\alpha C_S(\boA)\bou_S$  as a candidate for a numerical solution
and check how well it satisfies the system and what its error bound is.
We have
$$ 
\begin{array}{rcl}
\|\boA \tilde \bou -\bof\|_\infty
&=&
\|\boA 
\left(\bou^*+\alpha C_S(\boA)\bou_S\right) -\bof\|_\infty\\
&\leq &
\|\boA \bou^*-\bof\|_\infty+|\alpha|C_S(\boA)\|\boA\bou_S\|_\infty\\
&=& |\alpha|+ \|\boA \bou^*-\bof\|_\infty\\
&=& K(\boA)\|\boA \bou^*-\bof\|_\infty\\
\end{array}
$$
if we choose
$$
\alpha=(K(\boA)-1)\|\boA \bou^*-\bof\|_\infty.
$$ 
Thus $\tilde\bou$ is a valid candidate for numerical solving. 
The actual error is
\bql{eqsharpproof}.
\begin{array}{rcl}
\|\bou^*-\tilde\bou\|_\infty
&=&(K(\boA)-1)\|\boA \bou^*-\bof\|_\infty C_S(\boA)\\
&=&
(K(\boA)-1)C_S(\boA)\max_{1\leq k\leq N}|\lambda_k(u_j)-\tilde\lambda_k(u_j)|\\
&\geq &
(K(\boA)-1)C_S(\boA)|\lambda_j(u_j)-\tilde\lambda_j(u_j)|\\
&=&
(K(\boA)-1)C_S(\boA)\|u_j\|_S\|\boc\|_\infty\\
\end{array} 
\eq
proving the assertion.
\end{proof}\par\noindent
We shall come back to this worst-case construction in the examples
of Section \RSref{SecEx}.
%****************************************************************
\section{Dirichlet Problems}\RSlabel{SecDP}
The above error analysis simplifies for problems where 
Dirichlet values are given on boundary nodes, and where
approximations of differential operators are only needed in
interior points. Then we have $N$ approximations 
of functionals that are based on $M_I$ interior nodes
and $M_B$ boundary nodes, with $M=M_I+M_B$. 
We now use subscriots $I$ and $B$ to indicate vectors
of values on interior and boundary nodes, respectively. The linear system
now is
$$
\boB\bou_I=\bof_I-\boC\bog_B
$$
while the previous section dealt with the full system
$$
 \boA
\left(
\begin{array}{rcl}
\bou_I\\
\bou_B
\end{array}
\right)
=
\left(
\begin{array}{rl}
\bof_I\\ \bog_B
\end{array}
\right)\;\hbox{ with }
\boA=\left( 
\begin{array}{rl}
\boB & \boC\\
0 & \boI_B
\end{array}
 \right)
$$that has trivial approximations on the boundary.
Note that this splitting is standard practice in classical
finite elements when nonzero Dirichlet boundary conditions are given.
We now use the stability constant $C_S(\boB)$ for $\boB$, not for $\boA$,
and examples will show that it often comes out much smaller
than $C_S(\boA)$.
The consistency bounds \eref{eqconserr} stay the same, but they now 
take the form
$$
\|\boB\bou^*_I+\boC\bou_B^*-\bof_I\|_q
=\|\boB\bou^*_I+\boC\bog_B-\bof_I\|_q\leq \|\boc_I\|_q\|u^*\|_S.
$$ 
The numerical method should now guarantee
$$
\|\boB\tilde\bou_I+\boC\bog_B-\bof_I\|_q\leq K(\boB) 
\|\boB\bou_I^*+\boC\bog_B-\bof_I\|_q
$$ 
with a reasonable $K(\boB)\geq 1$. Then the same error analysis applies,
namely
$$
\begin{array}{rcl}
\|\bou^*_I-\tilde\bou_I\|_p
&\leq &
C_S(\boB)\|\boB(\bou^*_I-\tilde\bou_I)\|_q\\
&\leq &
C_S(\boB)\|\boB\bou^*_I-\boC\bog_B-\bof_I\|_q
+C_S(\boB)\|\boB\tilde\bou_I-\boC\bog_B-\bof_I)\|_q\\
&\leq &
C_S(\boB)(1+K(\boB))\|\boB\bou^*_I-\boC\bog_B-\bof_I\|_q\\
&\leq &
C_S(\boB)(1+K(\boB))\|\boc_I\|_q\|u^*\|_S.
\end{array} 
$$
%****************************************************************
\section{Consistency Analysis}\RSlabel{SecCA}
There are many ways to determine the {\em stiffness matrix
  elements}  $a_j(\lambda_k)$ arising in \eref{eqfAu}
and \eref{eqlamalu}, but they are either based on
{\em trial/shape functions} or on
{\em direct discretizations} as 
described in Section
\RSref{SecDis}.
We do not care here which technique is used.
As a by-product, our method will allow to compare 
different approaches on a fair basis.
\biglf
To make the constants $c(\lambda)$ in \eref{eqconserr} 
numerically accessible, we assume that the norm   $\|.\|_S$ 
comes from a Hilbert subspace $U_S$ of $U$ that has a reproducing kernel 
$$
K\;:\;\Omega\times\Omega \to \R.
$$
The squared norm of the error functional $\lambda-\tilde\lambda$
of the approximation $\tilde\lambda$ in \eref{eqlamalu} 
then is the value of the quadratic form
\bql{eqQQ}
\begin{array}{rcl}
Q^2(\lambda,\tilde\lambda)&:=& \|\lambda-\tilde\lambda\|_{U_S^*}^2\\
&=&
\lambda^x\lambda^yK(x,y)-
\displaystyle{2\sum_{j=1}^Ma_j(\lambda) \lambda_j^x\lambda^yK(x,y)  }\\
&&+\displaystyle{\sum_{j,k=1}^M 
a_j(\lambda)a_k(\lambda)\lambda_j^x\lambda_k^yK(x,y)  }  
\end{array}
\eq
which can be explicitly evaluated, though there will be
serious numerical cancellations because the result is small while the
input is not. It provides the explicit error bound
$$
\begin{array}{rcl}
|\lambda(u^*)-\tilde\lambda(u^*)|^2
&\leq & Q^2(\lambda,\tilde\lambda)\|u^*\|_S^2
\end{array}
$$ 
such that we can work with
$$
c(\lambda)=Q(\lambda,\tilde\lambda).
$$ 
As mentioned already, the quadratic form \eref{eqQQ} in its na\"ive form
has an unstable evaluation
due to serious cancellation. In \RScite{davydov-schaback:2015-1},
these problems were partly overcome by % RBF-QR methods and  
variable precision arithmetic, while the paper 
\RScite{larsson-et-al:2013-1} provides a very nice stabilization
technique, but unfortunately  confined to
approximations based on the Gaussian kernel. We hope to be able to
deal with stabilization of the evaluation 
of the quadratic form in a forthcoming paper. 
\biglf
On the positive side, there are cases where these instabilities do not occur,
namely for {\em polyharmonic kernels}. We shall come back to this 
in Section \RSref{SecAbPHK}.
\biglf
Of course, there are many {\em theoretical} results bounding 
the consistency error \eref{eqconserr}, 
e.g. \RScite{mirzaei-et-al:2012-1,davydov-schaback:2015-1} 
in terms of 
$\|u^*\|_S$, with explicit convergence orders in terms of powers of 
{\em fill distances}
$$
h:=\sup_{y\in\Omega}\min_{x_j}\|y-x_j\|_2.
$$
We call there orders {\em consistency orders} in what follows. 
Except for Section \RSref{SecAbPHK}, 
we do not survey such results here, but users can be sure that
a sufficiently fine fill distance and sufficient smoothness of the solution
will always lead to a high consistency order.  Since rates increase when more
nodes  are used, we target $p$-methods, not $h$-methods
in the language of the finite element literature, and we assume sufficient
regularity for this.
\biglf
Minimizing the quadratic form \eref{eqQQ} over the weights
$a_j(\lambda)$ yields
discretizations with {\em optimal} consistency with respect to
the choice of the space $U_S$
\RScite{davydov-schaback:2015-1}. But their calculation may be unstable
\RScite{larsson-et-al:2013-1}
and they usually lead to non-sparse matrices unless users
restrict the used nodes for each single functional. If they are combined with 
a best possible choice of trial functions, namely the Riesz representers
$v_j(x)=\lambda_j^yK(x,y)$ of the test functionals, the resulting 
linear system is symmetric and positive definite, provided that the
functionals are linearly independent. This method is
{\em symmetric collocation} \RScite{fasshauer:1997-1,%
franke-schaback:1998-2a,franke-schaback:1998-1}, and it is
an {\em optimal recovery} method in the space $U_S$
\RScite{schaback:2015-3}. It leads to non-sparse
matrices and suffers from severe instability, but it is error-optimal.
Here, we focus on non-optimal methods that allow sparsity. 
\biglf
Again, the instability of optimal approximations can be avoided using
polyharmonic kernels, and the next section will describe how this works.  
%****************************************************************
\section{Approximations by Polyharmonic Kernels}\RSlabel{SecAbPHK}
Assume that we are working in a context where we know that
the true solution $u^*$ lies in Sobolev space $W_2^m(\Omega)$ for
$\Omega \subset\R^d$, or, by Whitney extension
also in $W_2^m(\R^d)$.  Then the consistency error \eref{eqconserr}
of any given approximation should be evaluated in that space,
and taking an optimal approximation in that space would yield 
a system with optimal consistency.
\biglf
But since the evaluation and calculation of approximations in $W_2^m(\R^d)$
is rather unstable, a workaround is appropriate. Instead of the full norm
in  $W_2^m(\R^d)$ one takes the seminorm involving only the order $m$
derivatives. This originates from early work of Duchon \RScite{duchon:1979-1}
and leads to Beppo-Levi spaces instead of Sobolev spaces 
(see e.g. \RScite{wendland:2005-1}),
but we take a summarizing shortcut here.
Instead of the Whittle-Mat\'ern kernel reproducing $W_2^m(\R^d)$,
the radial {\em polyharmonic} kernel 
\bql{eqKmd}
H_{m,d}(r):=
\left\{
\begin{array}{ll}
(-1)^{\lceil m-d/2\rceil}r^{2m-d},& 2m-d \hbox{ odd}\\ 
(-1)^{1+m-d/2}r^{2m-d}\log r, &2m-d \hbox{ even} 
\end{array} 
\right\}
\eq
is taken, up to a scalar multiple
\bql{eqphfact}
\left\{
\begin{array}{ll}
\dfrac{\Gamma(m-d/2)}{2^{2m}\pi^{d/2}(m-1)!} & 2m-d \hbox{ odd}\\ 
\dfrac{1}{2^{2m-1}\pi^{d/2}(m-1)!(m-d/2)!}&2m-d \hbox{ even} 
\end{array} 
\right\}
\eq
that is used to match the seminorm in Sobolev space $W^m(\R^d)$.
We allow $m$ to be integer or half-integer.
This kernel is {\em conditionally positive definite}
of order $k=\lfloor m-d/2\rfloor+1$, and this has the consequence 
that approximations working in that space 
must be exact on polynomials of al least that order
(= degree plus one). In some sense, this is the price to be paid for
omitting the lower order derivatives in the Sobolev norm,
but polynomial exactness will turn out to be a good feature, not a bug. 
\biglf
As an illustration for the connection between 
the polyharmonic kernel $H_{m,d}(r)$ and the 
Whittle-Mat\'ern kernel
$K_{m-d/2}(r)r^{m-d/2}$ reproducing $W_2^m(\R^d)$, we state 
the observation that (up to
constants)
the polyharmonic kernel arises
as the first term in the expansion of the 
Whittle-Mat\'ern kernel that is not an even power of $r$. 
For instance, up to higher-order terms,
$$
K_{3}(r)r^{3}=16-2r^2+\frac{1}{4}r^4+\frac{1}{24}r^6\log(r)
$$
containing $H_{4,2}(r)=r^6\log(r)$ up to a constant. This seems to hold in
general for $K_n(r)r^n$ and $n=m-d/2$ for integer $n$ and even dimension $d$.
Similarly, 
$$
\frac{1}{\sqrt{2\pi}}K_{5/2}(r)r^{5/2}=3-\frac{1}{2}r^2+\frac{1}{8}r^4
-\frac{1}{15}r^5
$$
contains $H_{4,3}(r)=r^5$ up to a constant, and this generalizes to half-integer
$n$ with $n=m-d/2$. A rigid proof seems to be missing, but the upshot is that
the polyharmonic kernel, if written with $r=\|x-y\|_2$, differs from the
Whittle-Mat\'ern kernel only by lower-order polynomials and higher-order terms,
being simpler to evaluate.
\biglf
If we have an arbitrary approximation \eref{eqlamalu} that is exact
on polynomials of order $k$, we can insert its coefficients $a_j$ into
the usual quadratic form \eref{eqQQ} using the polyharmonic kernel there,
and evaluate the error. Clearly, the error is not smaller than the error of the
optimal approximation using the polyharmonic kernel, and let us denote the 
coefficients of the latter by $a_j^*$. 
\biglf
We now consider {\em scaling}. Due to shift-invariance, we can
assume that we have  a homogeneous
differential operator of order $p$ that is to be evaluated at the origin,
and we use scaled points $hx_j$ for its nodal approximation.
It then turns out \RScite{schaback:2015-1} 
that the optimal coefficients $a_j^*(h)$ scale like
$a_j^*(h)=h^{-p}a_j^*(1)$, and the quadratic form $Q$ of \eref{eqQQ}
written in terms
of coefficients  as
$$
\begin{array}{rcl}
Q^2(a)
&=&
\lambda^x\lambda^yK(x,y)-
\displaystyle{2\sum_{j=1}^Ma_j(\lambda) \lambda_j^x\lambda^yK(x,y)  }\\
&&+\displaystyle{\sum_{j,k=1}^M 
a_j(\lambda)a_k(\lambda)\lambda_j^x\lambda_k^yK(x,y)  }  
\end{array}
$$
scales {\em exactly} like 
$$
Q(a^*(h))=h^{2m-d-2p}Q(a^*(1)),
$$
proving that 
{\em there is no approximation of better order} in that space,
no matter how users calculate their approximation. Note that strong methods
(i.e. collocation) for second-order PDE problems \eref{eqDP} using functionals
\eref{eqlamLmuB} have $p=2$ while the weak functionals of \eref{eqlamweak}
have $p=1$. This is a fundamental difference between weak and strong
formulations, but note that it is easy to have methods of arbitrarily high 
consistency order.
\biglf
In practice, any set
of given and centralized nodes $x_j$ can be blown up to points $Hx_j$
of average pairwise distance 1. Then 
the error and the weights can be calculated for the blown-up situation, and
then the scaling laws for the coefficients and the error are applied
using $h=1/H$. This works for all scalings, without serious instabilities.
\biglf
Now that we know an optimal approximation with a simple and stable scaling,
why bother with other approximations? They will not have a smaller
worst-case consistency error, and they will not always have the scaling 
property  $a_j(h)=h^{-p}a_j(1)$, causing instabilities when evaluating the
quadratic form. If they do have that scaling law, then 
$$
Q(a(h))=h^{2m-d-2p}Q(a(1))\geq h^{2m-d-2p}Q(a^*(1))=Q(a^*(h))
$$
can easily be proven, leading to stable calculation for 
an error that is not smaller than the optimal one. In contrast
to standard results on the error of kernel-based approximations,
we have no restriction like $h\leq h_0$ here, since the scaling law 
is exact and holds for
all $h$.  
\biglf
If the smoothness $m$ for error evaluation is {\em fixed}, it will not pay off
to use approximations with higher orders of polynomial exactness, 
or using kernels with higher smoothness. They cannot beat the optimal
approximations for that smoothness class, and the error bounds of these
are sharp. Special approximations
can be better in a single case, but this paper deals
with worst-case bounds, and then the optimal approximations
are always superior.
\biglf
The optimal approximations can be calculated for small numbers of nodes,
leading to sparse stiffness matrices. One needs enough points to
guarantee polynomial exactness of order $k=\lfloor m-d/2\rfloor+1$.
The minimal number of points actually
needed will depend on their geometric placement. The five-point star
is an extremely symmetric 
example with exactness of order 4 in $d=2$, but this order will 
normally need 
15 points in general position because the dimension of the space of
third-degree polynomials in $\R^2$ is 15. 
\biglf
The upshot of all of this is that, given a fixed
smoothness $m$ and a dimension $d$,  polyharmonic stencils
yield sparse
optimal approximations that can be stably calculated and evaluated.
Examples are in \RScite{schaback:2015-1} and in Section 
\RSref{SecEx} below. See \RScite{iske:2003-1} for an early work 
on stability of 
interpolation by polyharmonic kernels, and \RScite{aboiyar-et-al:2010-1}
for an example of an advanced application.
%****************************************************************
\section{Stability Analysis}\RSlabel{SecSA}
We now take a closer look at the stability constant $C_S(\boA)$ from
\eref{eqCSC}. It can be rewritten as
\bql{eqCASuAu} 
C_S(\boA)=\displaystyle{\sup\{
\|\bou\|_p\;:\;\|\boA\bou\|_q\leq 1\}   }
\eq
and thus $2C_S(\boA)$ is the $p$-norm diameter of the convex set $\{
\bou\in\R^M\;:\;\|\boA\bou\|_q\leq 1\} $.
In the case $p=q=\infty$ that will be particularly important below, 
this set is a polyhedron, and the constant $C_S(\boA)$ can be calculated via linear
optimization. We omit details here, but note that the calculation tends
to be computationally unstable and complicated. It is left to
future research to provide a good estimation technique for the stability
constant
$C_S(\boA)$ like
MATLAB's {\tt condest} for estimating the $L_1$
condition number of a square matrix.
\biglf
In case $p=q=2$ we get
$$
C_S(\boA)^{-1}=\min_{1\leq j\leq M}\sigma_j
$$  
for the $M$ positive singular values $\sigma_1,\ldots,\sigma_M$ 
of $A$, and these are obtainable by
{\em singular value decomposition}.  
\biglf
To simplify the computation, one might calculate
the pseudoinverse $\boA^\dagger$ of $\boA$ and then take the standard
$(p,q)$-norm of it, namely
$$
\|\boA^\dagger\|_{p,q}:=\displaystyle{\sup_{\bov\neq 0}
\dfrac{\|\boA^\dagger\bov\|_p}{\|\bov\|_q}   }. 
$$
This overestimates $C_S(\boA)$ due to
$$
\|\boA^\dagger\|_{p,q}
\geq\displaystyle{\sup_{\bov=\boA\bou\neq 0}
\dfrac{\|\boA^\dagger\boA\bou\|_p}{\|\boA\bou\|_q} 
= \sup_{\bou\neq 0}
\dfrac{\|\bou\|_p}{\|\boA\bou\|_q}  }=C_S(\boA)
$$
since $C_S(\boA)$ is the norm of the pseudoinverse not on all of $\R^N$, but
restricted to
the $M$-dimensional range of $\boA$ in $\R^N$. Here, we again used that
$\boA$ has full rank, thus $\boA^\dagger\boA=I_{M\times M}$.
\biglf
Calculating the pseudoinverse may be as expensive as the numerical solution if
the system \eref{eqLinSys} itself, but if a user wants to have a close grip on
the error, it is worth while. It assures stability of the numerical process,
if not intolerably large, as we shall see.  Again, we hope for future research
to produce an efficient estimator.
\biglf
A simple possibility, restricted to square systems, is to use the fact that
MATLAB's {\tt condest} estimates the 1-norm-condition number, which is the
$L_\infty$ condition number of the transpose. Thus
\bql{eqCStilde}
\tilde C_S(\boA):=\dfrac{{\tt condest}(\boA')}{\|\boA\|_\infty}
\eq
is an estimate of the $L_\infty$ norm of $\boA^{-1}$. This is computationally 
very cheap for sparse matrices and turns out to work fine on the examples in
Section \RSref{SecEx},
but an extension to non-square matrices
is missing.
\biglf
We now switch to theory and want to show that users can expect 
$C_S(\boA)$ to be bounded above independent of the discretization details,
if the underlying problem is well-posed.  To this end, we 
use the approach of \RScite{schaback:2015-2} in what follows.
\biglf
Well-posed analytic problems of the form \eref{eqConPro}
allow a stable reconstruction of $u\in U$ 
from their full set of  data $f_\lambda(u),\;\lambda\in \Lambda$.
This {\em analytic stability} can often be
described as
\bql{equWP}
\|u\|_{WP}\leq C_{WP} \sup_{\lambda\in \Lambda}|\lambda(u)| \fa u\in U,
\eq
where the {\em well-posedness norm} $\|.\|_{WP}$ 
usually is weaker than the norm
$\|.\|_U$. For instance, elliptic second-order 
Dirichlet boundary value problems written in strong form satisfy 
\bql{equUCLuBufirst} 
\|u\|_{\infty,\Omega}\leq \|u\|_{\infty,\partial\Omega}+C \|L
u\|_{\infty,\Omega}
\fa u\in U:=C^2(\Omega)\cap C(\overline{\Omega}),
\eq
see e.g. \cite[(2.3), p. 14]{braess:2001-1}, and this is \eref{equWP} 
for $\|.\|_{WP}=\|.\|_{\infty}$.  
\biglf
The results of \RScite{schaback:2015-2} then show that 
for each trial space $U_M\subset U$ one can find a test set $\Lambda_N$ such
that \eref{equWP} takes a discretized form
$$
\|u\|_{\infty}\leq 2C_{WP} \sup_{\lambda_k\in \Lambda_N}|\lambda_k(u)| \fa u\in U_M,
$$
and this implies
$$
|u(x_j)|\leq 2C_{WP} \sup_{\lambda_k\in \Lambda_N}|\lambda_k(u)| \fa u\in U_M
$$
for all nodal values. 
%If the subspace $U_M$ has a Lagrange basis 
%$u_1,\ldots,u_M$ 
%with respect to $M$ nodes $x_1,\ldots,x_M$, we get
%$$
%|u_i(x_j)|\leq 2C_{WP} \sup_{\lambda_k\in \Lambda_N}|\lambda_k(u_i)|,\;1\leq
%i,j\leq M,
%$$
This proves  % proving 
a uniform 
stability property of the stiffness matrix with entries $\lambda_k(u_i)$.
The functional approximations in \RScite{schaback:2015-2} were of the form
$a_j(\lambda)=\lambda(u_j)$, and then
$$
\begin{array}{rcl}
\|\bou\|_\infty
&\leq& 
2C_{WP} \sup_{\lambda_k\in \Lambda_N}|\lambda_k(u)|\\
&=&
2C_{WP} \sup_{\lambda_k\in \Lambda_N}|\lambda_k\left(\sum_{i=1}^Mu(x_i)u_i\right)|\\
&=&
2C_{WP} \sup_{\lambda_k\in \Lambda_N}|\sum_{i=1}^Mu(x_i)\lambda_k(u_i)|\\
&=&
2C_{WP} \|\boA\bou\|_\infty
\end{array} 
$$
and thus
$$
C_S(\boA) \leq 2C_{WP}.
$$
This is a prototype situation encouraging users to expect reasonably
bounded norms of the pseudoinverse, provided that the 
norms are properly chosen. 
\biglf
However, the situation of 
\RScite{schaback:2015-2} is much more special than here, because it is confined
to the trial function approach.
While we do not even specify trial spaces here, 
the paper  \RScite{schaback:2015-2} 
relies on the condition $a_j(\lambda)=\lambda(u_j)$ for a Lagrange basis
of a trial space, i.e. exactness of the approximations on a chosen trial space.
This is satisfied in nodal methods based on trial spaces, 
but not in direct nodal methods. In particular, it works 
for Kansa-type collocation and MLS-based nodal meshless methods,
but not for localized kernel approximations and direct MLPG techniques
in nodal form.
\biglf
For general choices of  $a_j(\lambda)$, the stability problem
is a challenging research area that is not addressed here. 
Instead, users are asked 
to monitor the row-sum norm of the pseudoinverse numerically and apply
error bounds like \eref{eqthirdstep} for $p=q=\infty$. Note that the choice 
of discrete $L_\infty$ norms is dictated by the well-posedness inequality
\eref{equUCLuBufirst}. As pointed out above, chances are good
to observe numerical stability for well-posed problems,
provided that test functionals are chosen properly.
We shall see this in the examples of Section \RSref{SecEx}.
In case of square stiffness matrices, users can apply \eref{eqCStilde}
to get a cheap and fairly accurate estimate of the stability constant. 
\biglf
For problems in weak form, the well-posedness norm usually
is not $\|.\|_{\infty,\Omega}$ but $\|.\|_{L_2(\Omega)}$, and then we 
might get into 
problems using a nodal basis. In such cases, an $L_2$-orthonormal basis
would be needed for uniform stability, but we refrain from considering weak
formulations here.  
%****************************************************************
\section{Examples}\RSlabel{SecEx}
In all examples to follow, the nodal points are $x_1,\ldots,x_M$
in the domain $\Omega=[-1,+1]^2\subset\R^2$, and parts of them are placed on the boundary.
We consider the standard Dirichlet problem for the Laplacian throughout,
and use testing points $y_1,\ldots, y_n\in \Omega$ for the Laplacian  and
$z_1,\ldots,z_k\in \partial\Omega$ 
for the Dirichlet boundary data
in the sense of \eref{eqlamLmuB}. 
Note that in our error bound \eref{eqthirdstep}
the right-hand sides of problems like \eref{eqDP} do not occur at all.
This means that everything is only dependent on how the discretization works,
it does not depend on any specific choice of $f$ and $g$. 
\biglf
We omit detailed
examples that show how the stability constant $C_S(\boA)$ decreases when 
increasing the number $N$ of test functionals. An example is in
\RScite{schaback:2015-2}, and  \eref{eqCASuAu} shows that stability
must improve if rows are added to $\boA$. Users are urged to make sure that 
their approximations \eref{eqlamku}, 
making up the rows of the stiffness matrix,
have roughly the same consistency order,
because adding 
equations will then  improve stability
without serious change of the consistency error.
\biglf
We first take regular points on a 2D grid of sidelength $h$ in
$\Omega=[-1,+1]^2\subset \R^2$ and interpret all points as nodes.
On interior nodes, we approximate the Laplacian by the usual five-point star
which is exact on polynomials 
up to degree 3 or order 4. On boundary nodes, we take the 
boundary values as given. This yields a square linear system. Since the
coefficients of the 
five-point star blow up like ${\cal O}(h^2)$ for $h\to 0$, 
the row-sum norm of $\boA$ 
and the condition must blow up like ${\cal O}(h^{-2})$, which can easily be
observed. The pseudoinverse does not blow up since the Laplacian part of 
$\boA$ just takes means and the boundary part is the identity.
For the values of $h$ we computed, its norm was bounded by roughly 1.3.
This settles the stability issue from a practical point of view.
Theorems on stability are not needed.
\biglf
Consistency depends on the regularity space $U_S$ chosen. We have 
a fixed classical discretization strategy via the five-point star, 
but we can evaluate the consistency error in different spaces. 
Table \RSref{tab5} shows the results for  Sobolev space $W_2^4(\R^d)$.
It clearly shows linear convergence, and its last column
has the major part of the worst-case relative error bound 
\eref{eqthirdstep}. The estimate $\tilde C_S(\boA)$ from \eref{eqCStilde}
agrees with $C_S(\boA)$ to all digits shown. 
Note that for all methods that need
continuous point evaluations of the Laplacian in 2D,
one cannot work with less smoothness, because the Sobolev inequality 
requires $W_2^m(\R^2)$ with $m>2+d/2=3$. The arguments in Section
\RSref{SecAbPHK} show that the consistency order then is at most 
$m-d/2-p=m-3=1$, as observed. Table \RSref{tab5B} shows 
the improvement 
if one uses the partial matrix $\boB$ of Section \RSref{SecDP}. 
\biglf

\begin{table}[hbt]\centering
\begin{tabular}{||r||r|r|r|r||}%\centering
\hline
$M=N$ & $h$ & $C_S(\boA)$ & $\|\boc\|_\infty$ & $C_S(\boA)\,\|\boc\|_\infty$\\
\hline
25 & 0.5000 & 1.281250 & 0.099045 & 0.126901 \\
81 & 0.2500 & 1.291131 & 0.051766 & 0.066837 \\
289 & 0.1250 & 1.293783 & 0.026303 & 0.034030 \\
1089 & 0.0625 & 1.294459 & 0.013222 & 0.017116 \\
\hline
\hline
\end{tabular}
\caption{Results for five-point star on the unit square, for 
$W_2^4(\R^2)$ and the full matrix $\boA$\RSlabel{tab5}}
\end{table}
\begin{table}[hbt]\centering
\begin{tabular}{||r||r|r|r|r||}%\centering
\hline
$M_I=N_I$ & $h$ & $C_S(\boB)$ & $\|\boc_I\|_\infty$ & $C_S(\boB)\,\|\boc_I\|_\infty$\\
\hline
9 & 0.5000 & 0.281250 &   0.099045 & 0.027856 \\
49 & 0.2500 & 0.291131 &   0.051766 & 0.015071 \\
225 & 0.1250 & 0.293783 &   0.026303 & 0.007727 \\
961 & 0.0625 & 0.294459 &   0.013222 & 0.003893 \\
\hline
\hline
\end{tabular}
\caption{Results for five-point star on the unit square, for 
$W_2^4(\R^2)$ and the partial matrix $\boB$\RSlabel{tab5B}}
\end{table}

\biglf  
We now demonstrate
the sharpness of our error bounds. We implemented the 
construction of Section \RSref{SecSh} for $K(\boA)=2$ 
and the situation in the final
row of Table \RSref{tab5}. This means that, given $\boA$, 
we picked values of $f$ and $g$ 
to realize worst-case stability and consistency, with known
value vectors $\bou^*$ and $\tilde\bou$.
Figure \RSref{figStabConWC} shows the values of $\bou_S$ and $\bou_j=\bou^*$
in the notation of the proof of Theorem \RSref{thesharp},
while Figure \RSref{figWC} displays $\tilde\bou$. 
The inequality 
\eref{eqsharp} is in this case 
$$
0.000226=C_S(\boA)\|u^*\|_S\|\boc\|_\infty\leq
\|\bou^*-\tilde\bou\|_\infty=0.000226\leq 3C_S(\boA)\|u^*\|_S\|\boc\|_\infty=0.000679
$$
and the admissibility inequality \eref{eqAufKAuf} is exactly satisfied with 
$K(\boA)=2$.
Even though this example is worst-case, the residuals 
and the error $\|\bou^*-\tilde\bou\|_\infty$ are small
compared to the last line of Table \RSref{tab5}, and users might 
suspect that the table has a useless
overestimation of the error. But the 
explanation  is that the above
bounds are absolute, not relative, while the norm of the true solution is 
$\|u^*\|_S=\|\boc\|_\infty=0.0132$. 
The relative form of the above bound is
$$
0.0171=
\dfrac{\|\bou^*-\tilde\bou\|_\infty}{\|u^*\|_S}\leq
0.0513,
$$
showing that the 
relative error bound 0.0171 
in Table \RSref{tab5} is attained by a specific example.
Thus our error estimation technique
covers this situation well. The lower bound in the worst-case construction
is attained because
this example  has equality in  \eref{eqsharpproof}.
\biglf
Note that our constructed case combines 
worst-case consistency with worst-case
stability, but in practical situations these two worst cases will
rarely happen at the same time. Figure \RSref{figStabConWC} 
shows that the worst case for stability seems to be 
a discretization of a discontinuous function, and 
therefore it may be that 
practical situations are systematically far away from the worst case.
This calls for a redefinition of the stability constant by 
restricting the range
of $\boA$ in an appropriate way. The worst case for stability arises 
for vectors of nodal values that are close to the eigenvector of the smallest
eigenvalue of $\boA$, but the worst case for consistency might 
systematically have small
inner products with eigenvectors for small eigenvalues.  

\begin{figure}[hbt]%\RSlabel{figsparse}
\begin{center}
\includegraphics[width=5cm,height=5cm]{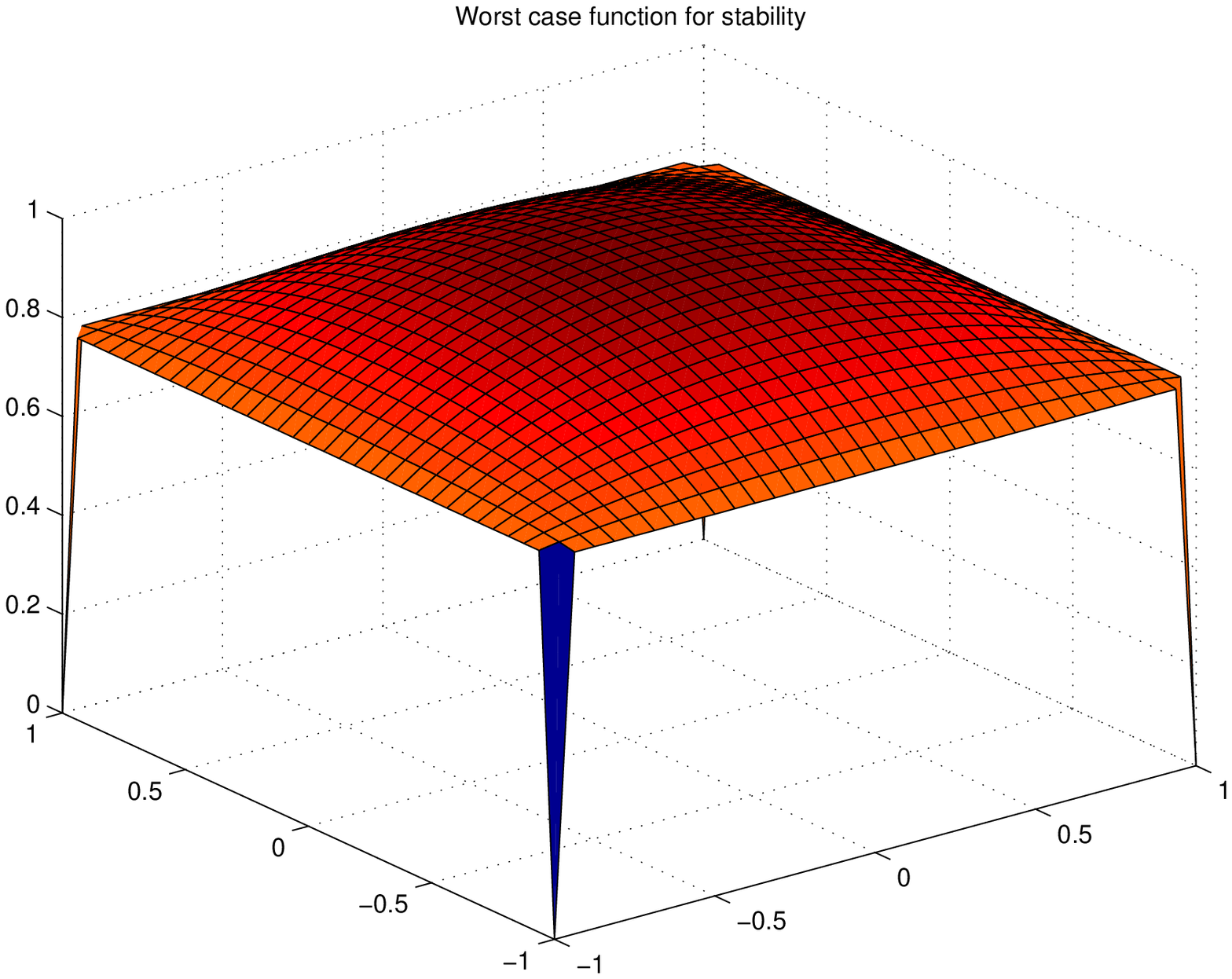}
%\caption{Stability worst case\RSlabel{figStaWC}}
%\end{center} 
%\end{figure} 
%\begin{figure}[hbt]%\RSlabel{figsparse}
%\begin{center}
\includegraphics[width=5cm,height=5cm]{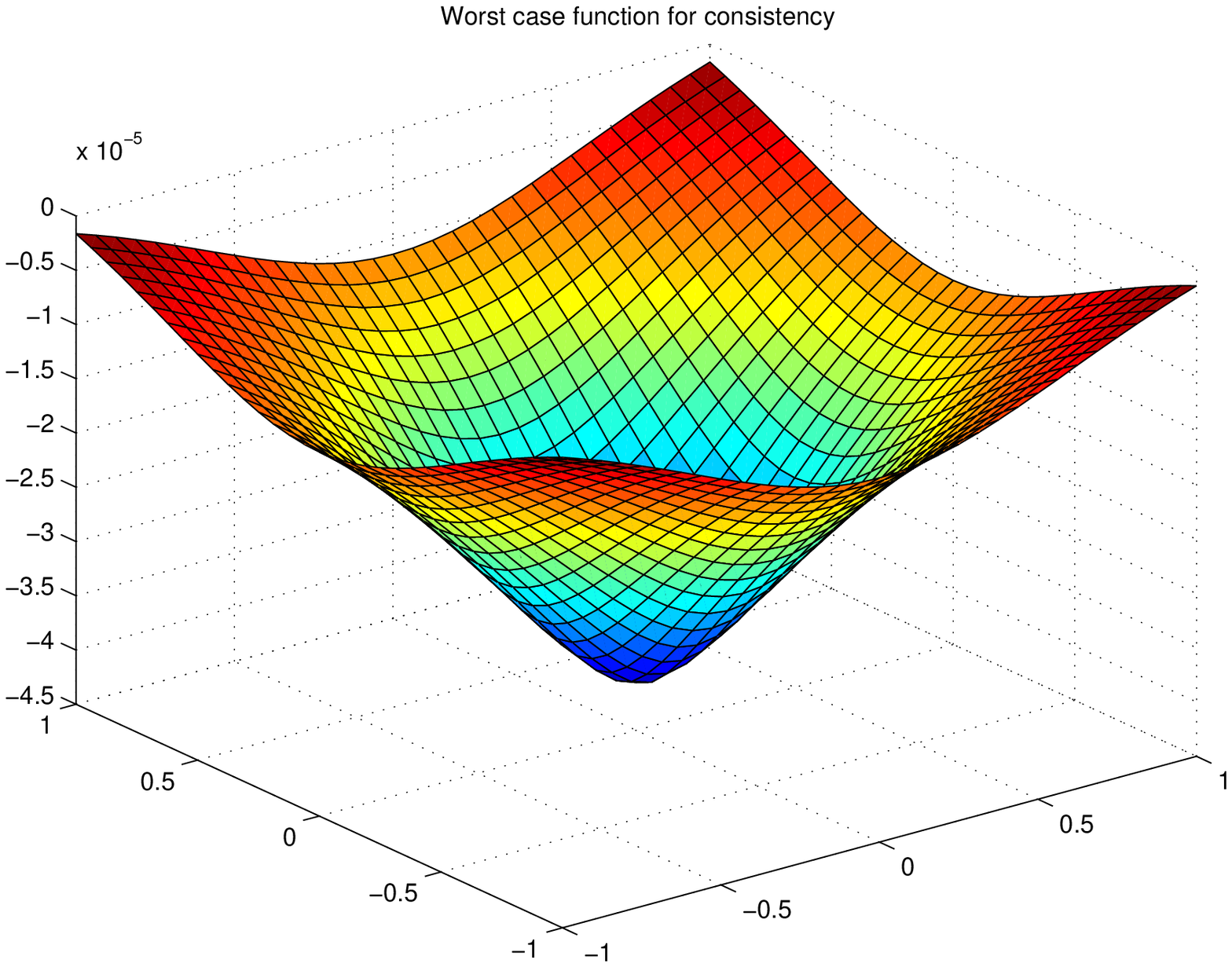}
\caption{Stability and consistency worst case\RSlabel{figStabConWC}}
\end{center} 
\end{figure} 
 
\begin{figure}[hbt]%\RSlabel{figsparse}
\begin{center}
\includegraphics[width=5cm,height=5cm]{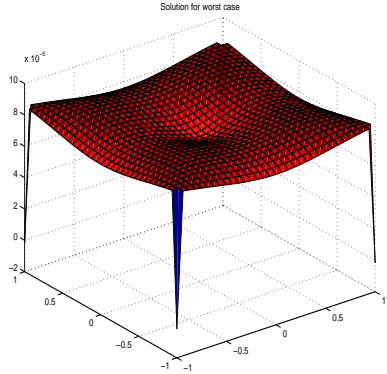}
\caption{Solution for joint worst case\RSlabel{figWC}}
\end{center} 
\end{figure} 

If we take the polyharmonic kernel $H_{4,2}(r)=r^6 \log r$ (up to a constant),
the five-point star is unique and therefore optimal, with consistency order 1,
see Section \RSref{SecAbPHK}.  This means that for given smoothness 
order $m=4$ and gridded nodes,
the five-point star already has the optimal
convergence order. Taking approximations of the Laplacian using
larger subsets of nodes might be exact on higher-order polynomials,
and will have smaller factors if front of the scaling law,
but the consistency and convergence {\em order}  
will not be better, at the expense
of losing sparsity. 
\biglf
To see how much consistency can be gained by using non-sparse
optimal approximations by polyharmonic kernels, we worked at $h=1$,
approximating the error of the Laplacian at the origin by 
data in the integer nodes 
$(m,n)$ with $-1\leq m,n\leq K$ for increasing $K$. This models the case
where the Laplacian is approximated in a near-corner point of the square.
Smaller $h$ can be handled by the scaling law.
The consistency error in $W_2^4(\R^2)$ goes down from 0.07165  to 0.070035  
when going from
25 to 225 neighbors (see Figure 
\RSref{figConsErr}), while 0.08461 is the error of the five-point star
at the origin. The gain is not worth the effort. The optimal stencils
decay extremely quickly away from the origin. This is predicted by
results of \RScite{matveev:1992-1} 
concerning exponential decay of Lagrangians of polyharmonic
kernels, as used successfully in \RScite{hangelbroek-et-al:2015-1}
to derive local inverse estimates. See \RScite{rabut:1992-2} 
for an early reference on polyharmonic 
near-Lagrange functions.

\begin{figure}[hbt]%\RSlabel{figsparse}
\begin{center}
\includegraphics[width=5cm,height=5cm]{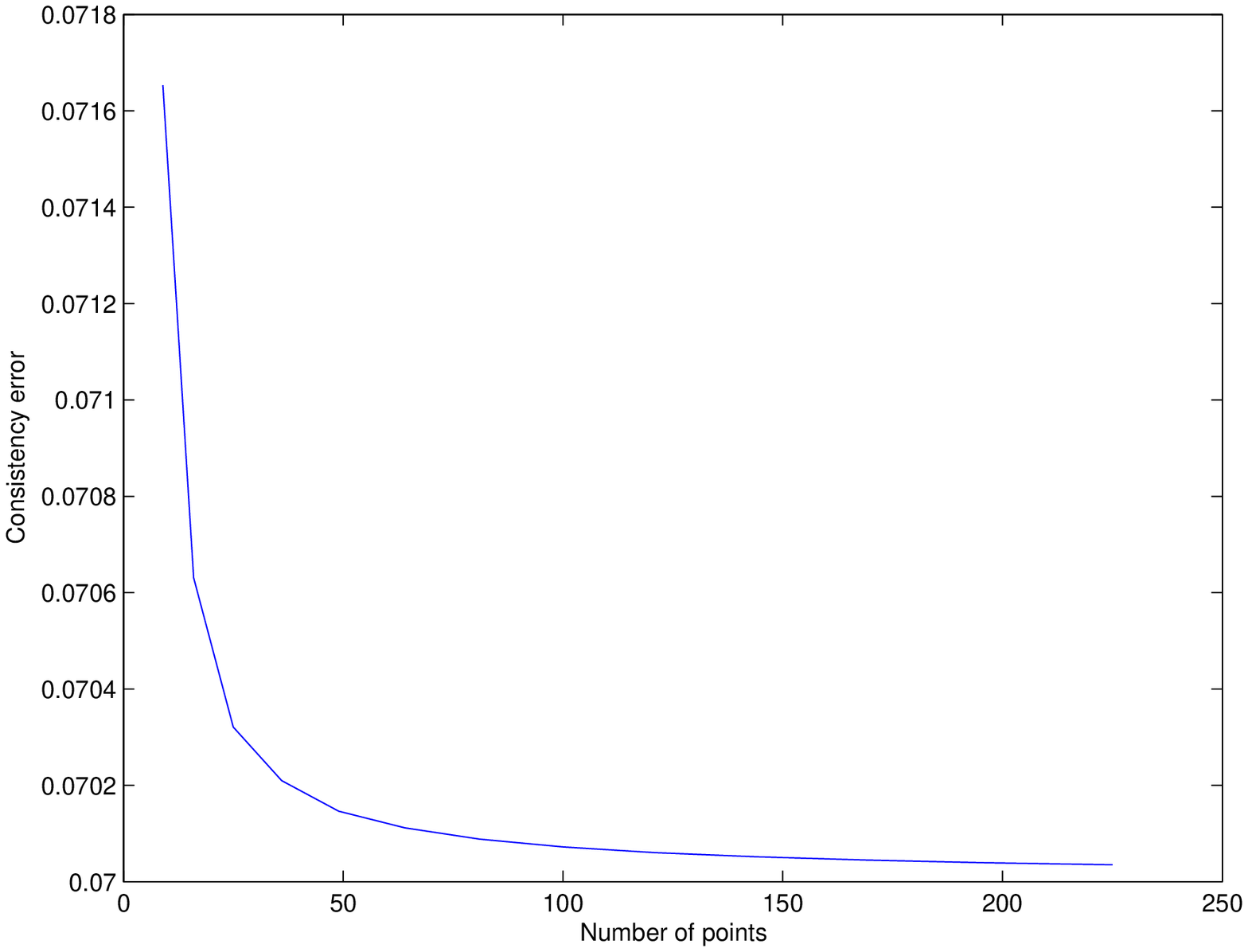}
\includegraphics[width=5cm,height=5cm]{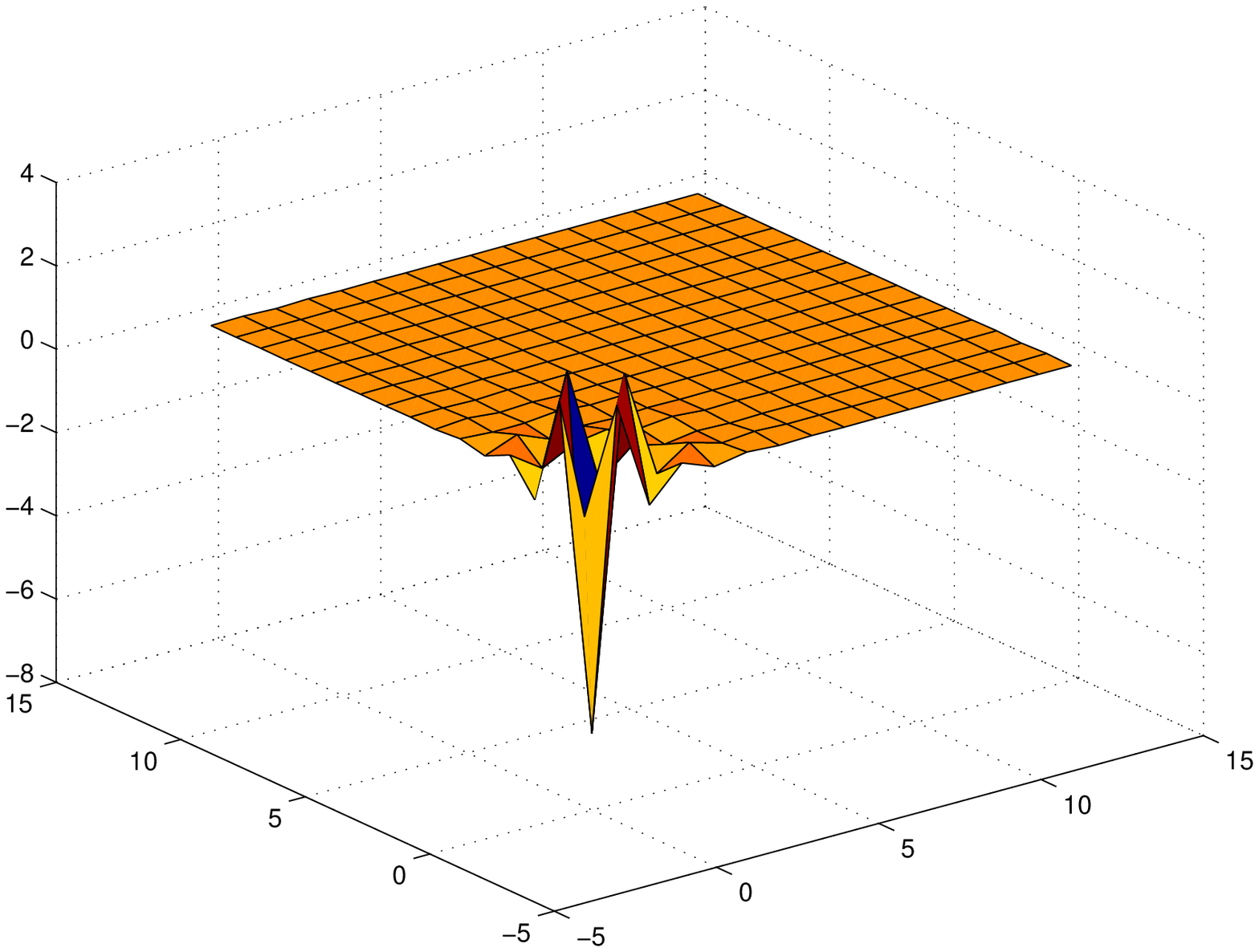}
\caption{Consistency error as a function of points offered,
and stencil 
of optimal approximation for 225 nodes, as a function on the nodes%
\RSlabel{figConsErr}}
\end{center} 
\end{figure} 
\biglf
We now show how the technique of this paper can be used to compare 
very different discretizations, while a smoothness order $m$
is fixed, in the sense that the true solution lies in 
Sobolev space $W_2^m(\Omega)$. Because we have $p$-methods in mind, we take
$m=6$ for the standard Dirichlet problem for the Laplacian in 2D 
and can expect an optimal  
consistency order $m-d/2-2=3$ for a strong discretization. 
Weak discretizations will be at least 
one order better, but we omit such examples. The 
required order of polynomial exactness
when using the polyharmonic kernel is $1+m-d/2=6$, which means 
that one should use at least 21 nodes for local approximations,
if nodes are in general position, without symmetries. The bandwidth of the 
generalized stiffness matrix must therefore be at least 21.
For convenience, we go to the unit square and a regular grid of
meshwidth $h$ first, to define the nodes.
But then we add uniformly distributed noise of $\pm h/4$
to each interior node, keeping the boundary nodes. 
Then we approximate the Laplacian at each interior node
locally by taking $n\geq 25$ nearest neighbor nodes, including boundary nodes,
and set up the reduced 
generalized square stiffness matrix  $\boB$ using the optimal
polyharmonic approximation based on these neighboring nodes. On the boundary, we
keep the given Dirichlet boundary values, following Section \RSref{SecDP}.
\biglf
Table \RSref{tabker30}
shows results for local optimal approximations based on the polyharmonic kernel 
$H_{6,2}(r)=r^{10}\log r$ and $n=30$ nearest neighbors. The stability constant
was estimated via \eref{eqCStilde}, for convenience and efficiency.
One cannot expect to see
an exact $h^3$ behavior in the penultimate column, since the nodes are randomly 
perturbed, but the overall behavior of the error is quite satisfactory. 
The computational complexity is roughly ${\cal O}(Nn^3)$, and note that 
the linear system is not solved at all, because we used MATLAB's {\tt condest}. 
\begin{table}[hbt]
\begin{tabular}{||r|r|r||r|r|r||}
\hline\hline
$N=M$ & $N_I=M_I$ & $h$ & $\tilde C_S(\boB)$ & $\|\boc_I\|_\infty$ & 
$C_S(\boB)\|\boc_I\|_\infty$ \\ \hline
  81 &   49 &  0.2500 &   2.3244 & 0.00075580 & 0.00175682 \\
 289 &  225 &  0.1250 &   0.3199 & 0.00005224 & 0.00001671 \\
1089 &  961 &  0.0625 &   0.2964 & 0.00000872 & 0.00000259 \\
4225 & 3969 &  0.0313 &   0.2961 & 0.00000147 & 0.00000044 \\
\hline\hline
\end{tabular} 
\caption{Optimal polyharmonic approximations using 30 neighbors
\RSlabel{tabker30}}
\end{table}
Comparing with Table \RSref{tabker25}, it pays off to use a few 
more neighbors, and this also avoids instabilities. Users unaware of
instabilities might think they can expect a similar behavior as in Table
\RSref{tabker30} when taking only 25 neighbors, but the third row of Table 
\RSref{tabker25} should teach them otherwise. By resetting
the random number generator, all tables 
were made to work on the same total set of points,
but the local approximations still yield rather different results.
\biglf 
The computationally cheapest way to calculate approximations with the
required polynomial exactness of order 6 on 25 neighbors is to solve
the linear $20\times 25$ system describing polynomial exactness 
via the MATLAB backslash operator. It will return a solution based on 21 points
only, i.e. with minimal bandwidth, but the overall behavior
in Table \RSref{tabpolback25} may not be worth the computational savings,
if compared to the optimal approximations on 30 neighbors.
\begin{table}[hbt]
\begin{tabular}{||r|r|r||r|r|r||}
\hline\hline
$N=M$ & $N_I=M_I$ & $h$ & $\tilde C_S(\boB)$ & $\|\boc_I\|_\infty$ & 
$C_S(\boB)\|\boc_I\|_\infty$ \\ \hline
  81 &   49 &  0.2500 &   8.0180 & 0.00318328 & 0.02552351 \\
 289 &  225 &  0.1250 &  66.7176 & 0.00039055 & 0.02605641 \\
1089 &  961 &  0.0625 & 417.8094 & 0.00003877 & 0.01620053 \\
4225 & 3969 &  0.0313 &  75.5050 & 0.00000663 & 0.00050082 \\
\hline\hline
\end{tabular} 
\caption{Optimal polyharmonic approximations using 25 neighbors
\RSlabel{tabker25}}
\end{table} 

\begin{table}[hbt]
\begin{tabular}{||r|r|r||r|r|r||}
 \hline\hline
$N=M$ & $N_I=M_I$ & $h$ & $\tilde C_S(\boB)$ & $\|\boc_I\|_\infty$ & 
$C_S(\boB)\|\boc_I\|_\infty$ \\ \hline
  81 &   49 &  0.2500 &   9.0177 & 0.00354151 & 0.03193624 \\
 289 &  225 &  0.1250 &  25.6153 & 0.00058952 & 0.01510082 \\
1089 &  961 &  0.0625 &  73.9273 & 0.00005482 & 0.00405249 \\
4225 & 3969 &  0.0313 &  19.6458 & 0.00001186 & 0.00023305 \\
\hline\hline
\end{tabular} 
\caption{Backslash approximation on 25  neighbors
\RSlabel{tabpolback25}}
\end{table} 

\biglf
A more sophisticated kernel-based {\em greedy} technique 
\RScite{schaback:2014-2,schaback:2015-1} 
uses between 21 and 30 points and works its way
through the offered 30 neighbors to find a compromise
between consistency error and support size. Table
\RSref{tabgreedy30} shows the results, with an average of
23.55 neighbors actually used. 

\begin{table}[hbt]
\begin{tabular}{||r|r|r||r|r|r||}
\hline\hline
$N=M$ & $N_I=M_I$ & $h$ & $\tilde C_S(\boB)$ & $\|\boc_I\|_\infty$ & 
$C_S(\boB)\|\boc_I\|_\infty$ \\ \hline
  81 &   49 &  0.2500 &   3.6188 & 0.00104016 & 0.00376411 \\
 289 &  225 &  0.1250 &   0.6128 & 0.00006821 & 0.00004180 \\
1089 &  961 &  0.0625 &   0.3061 & 0.00000961 & 0.00000294 \\
4225 & 3969 &  0.0313 &   0.2980 & 0.00000123 & 0.00000037 \\
\hline\hline
\end{tabular} 
\caption{Greedy polyharmonic approximations using at most 30  neighbors
\RSlabel{tabgreedy30}}
\end{table}  

For these examples, one can plot the consistency error as a function 
of the nodes, and there usually is a factor of 5 to 10 between
the error in the interior and on the boundary. Therefore
it should be better to let the node density increase towards the boundary,
though this may lead to instabilities that may call for overtesting,
i.e. to use $N >> M$. For the same $M$ and $N$ as before, but with
Chebyshev point distribution, see Table \RSref{tabCheb}. The additive noise on
the interior points was 0.01, and we used the greedy method for up to 30
neighbors. This leads to a larger bandwidth near the corners, and to a
consistency error that is now small at the boundary, see Figure
\RSref{figCheb}. The average number of neighbors used was 23.3.
Unfortunately, the scaling laws of stencils go down the drain here,
together with the proven
consistency order, but the results are still unexpectedly good. 
\begin{table}[hbt]
\begin{tabular}{||r|r|r||r|r|r||}
\hline\hline
$N=M$ & $N_I=M_I$ & $h$ & $\tilde C_S(\boB)$ & $\|\boc_I\|_\infty$ & 
$C_S(\boB)\|\boc_I\|_\infty$ \\ \hline
  81 &   49 &  0.2500 & 111.1016 & 0.00433490 & 0.48161488 \\
 289 &  225 &  0.1250 &   0.4252 & 0.00006541 & 0.00002781 \\
1089 &  961 &  0.0625 &   1.2133 & 0.00000677 & 0.00000821 \\
4225 & 3969 &  0.0313 &   0.4353 & 0.00000120 & 0.00000052 \\
\hline\hline
\end{tabular} 
\caption{Greedy polyharmonic approximations using at most 30  neighbors,
but in Chebyshev node arrangement
\RSlabel{tabCheb}}
\end{table}  

\begin{figure}[hbt]%\RSlabel{figsparse}
\begin{center}
\includegraphics[width=5cm,height=5cm]{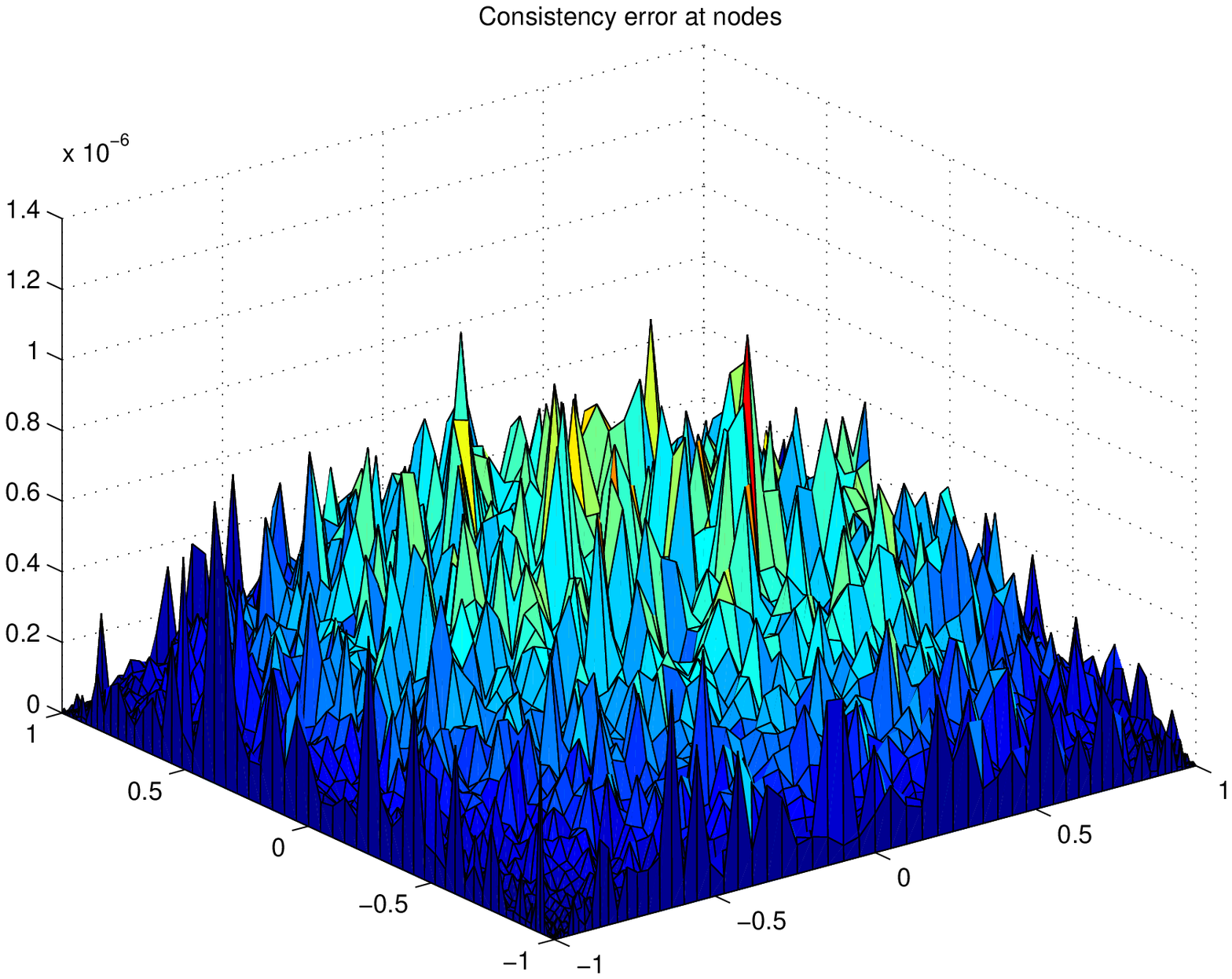}
%\caption{Stability worst case\RSlabel{figStaWC}}
%\end{center} 
%\end{figure} 
%\begin{figure}[hbt]%\RSlabel{figsparse}
%\begin{center}
\includegraphics[width=5cm,height=5cm]{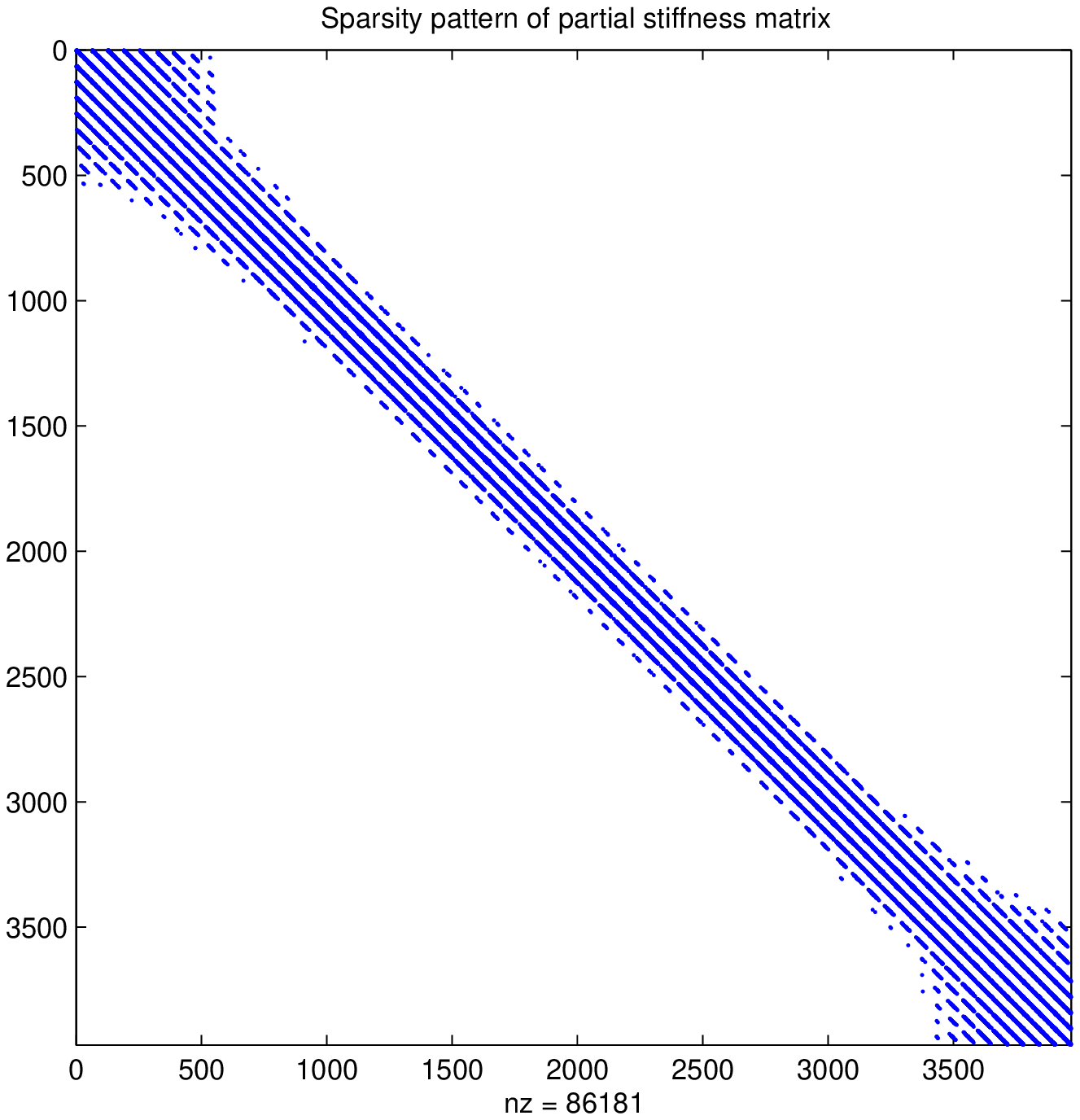}
\caption{Consistency plot and stiffness matrix $\boB$ 
for Chebyshev situation 
\RSlabel{figCheb}}
\end{center} 
\end{figure} 
\biglf
For reasons of space and readability, we provide no examples for
local approximations to weak functionals, and no comparisons with
local approximations obtained via Moving Least Squares or the
Direct Meshless Petrov Galerkin Method.  
%****************************************************************
\section{Conclusion and Outlook}\RSlabel{SecCO}
The tables of the preceding section show that the numerical calculation of
relative error bounds for PDE solving 
in spaces of fixed Sobolev smoothness can be done efficiently 
and with good results. This provides a general tool to evaluate discretizations
in a worst-case scenario, without referring to single examples 
and complicated theorems. Further examples should compare 
a large variety of competing techniques, 
the comparison being fair here as long as
the smoothness $m$ is fixed. 
\biglf
Users are strongly advised to
use the cheap stability estimate \eref{eqCStilde} {\bf anytime} 
to assess the stability of their discretization, if they have a square
stiffness matrix. And, if they are 
not satisfied with the final accuracy,  they should evaluate and plot the
consistency error like in Figure \RSref{figCheb} to see where the discretization
should be refined.   For all of this, polyharmonic 
kernels are an adequate tool.   
\biglf
It is left to
future research to investigate and 
improve the stability estimation technique via \eref{eqCStilde},
and, if the effort is worth while, to prove general theorems 
on sufficient criteria for stability. These will include 
assumptions on the placement of the trial nodes, as well as on
the selection of sufficiently many and well-placed test functionals.
In particular, stabilization by overtesting should work in general, but 
the examples in this paper show that overtesting may not be necessary at all.   
However, this paper serves as a practical workaround, 
as long as there are no theoretical cutting-edge results available.
%****************************************************************
\section*{Acknowledgement}\RSlabel{SecAck}
This work was strongly influenced by helpful discussions and e-mails with
Oleg Davydov and Davoud Mirzaei.
\bibliographystyle{plain}

\end{document}